\documentclass[12pt,oneside]{elsarticle}

      \usepackage{amssymb,color}
      \usepackage{amsmath}

      \newtheorem{theorem}{Theorem}

      \newtheorem{lemma}[theorem]{Lemma}
      \newtheorem{conjecture}[theorem]{Conjecture}
      \newtheorem{corollary}[theorem]{Corollary}
      \newtheorem{proposition}[theorem]{Proposition}
     \newtheorem{prop}[theorem]{Proposition}

      \newproof{proof}{Proof}
      \newproof{poc}{Proof of Corollary \ref{maincorollary}}
      \newproof{pot}{Proof of Theorem \ref{maintheorem}}

      \def \l {\lambda}
\def \Q {\mathbb Q}
\def \Z {\mathbb Z}

      \makeatletter
      \def\@setcopyright{}
      \def\serieslogo@{}
      \makeatother

\begin{document}

\author[iastate]{Jonas Kibelbek\corref{cor1}}
\ead{kibelbek@iastate.edu}
\author[iastate,cornell]{Ling Long}
\ead{linglong@iastate.edu and LL637@cornell.edu}
\author[iastate]{Kevin Moss}\ead{kmoss@iastate.edu}
\author[iastate]{Benjamin  Sheller}
\ead{bsheller@iastate.edu}
\author[iastate]{Hao Yuan}\ead{hyuan@iastate.edu}
\address[iastate]{Department of Mathematics, Iowa State University, 
Ames, 50011, USA}
\address[cornell]{ Mathematics Department, Cornell University, 
Ithaca, NY 14850, USA}

\cortext[cor1]{Corresponding author, Phone: 1-515-294-8150}


   \title{Supercongruences and Complex Multiplication}

  \begin{abstract}
   We study congruences involving truncated hypergeometric series of the form \linebreak 
    ${}_rF_{r-1}(\begin{smallmatrix}1/2, \, \cdots, 1/2 \\ 1, \, \cdots, 1 \end{smallmatrix}; \lambda)_{(mp^s-1)/2}
    = \sum_{k=0}^{(mp^s-1)/2} ((1/2)_k/k!)^r \lambda^k$ where $p$ is a prime and $m, s, r$ are positive integers. 
	These truncated hypergeometric series are related to the arithmetic of a family of
    algebraic varieties and exhibit Atkin and Swinnerton-Dyer type congruences. 
    In particular, when $r=3$, they are related to K3 surfaces. For special values of $\l$, with $s=1$ and $r=3$, 
    our congruences are stronger than what can be
    predicted by the theory of formal groups because of the presence of elliptic curves with complex multiplications.   
    They generalize a conjecture made by Rodriguez-Villegas for the $\l=1$ case and confirm some
    other supercongruence conjectures  at special values of $\l$. 
  \end{abstract}


   \maketitle

\section{Introduction}

The hypergeometric series $_rF_{r-1}$ is defined as
\[_{r}F_{r-1}\left(\begin{array}{c}
a_{1},a_{2},\cdots,a_{r}\\
b_{1},b_{2},\cdots,b_{r-1}
\end{array};\lambda\right):=\sum_{k=0}^{\infty} \left( \frac{(a_1)_k (a_2)_k \cdot \cdot \cdot
(a_r)_k}{k! (b_1)_k (b_2)_k \cdot \cdot \cdot (b_{r-1})_k} \right)  \l^{k} \]
where  $(a)_k:=a(a+1) \cdots (a+k-1)$ and none of the $b_i$ is a negative integer \cite{andrews}.
The truncated hypergeometric series $_rF_{r-1}(\begin{smallmatrix} a_1, \; \cdot \cdot \cdot,
\; a_r \\ b_1, \; \cdot \cdot \cdot, \; b_{r-1} \; \end{smallmatrix}; \l)_n$, is the degree $n$ polynomial in $\l$ obtained by 
truncating the hypergeometric series to the sum from $k=0$ to $n$.

In this paper, we study the arithmetic of 
$F_{r}(\l)_n:={}_rF_{r-1}(\begin{smallmatrix}\frac{1}{2}, \, \cdots, \frac{1}{2} \\ 1, \, \cdots,  1 \end{smallmatrix}; 
\lambda)_{n}$; these values are related to the  varieties $\mathcal X_r(\l): W^2=X_1\cdots X_r(X_1-X_2)\cdots(X_{r-1}-X_r)(X_r-\l X_1)$, 
which generalize  the Legendre family of elliptic curves.  In terms of counting points, Deuring's argument \cite[pp. 255]{Deuring} 
yields that for any $\l\in \mathbb F_p$, $$\#\mathcal X_r(\l)(\mathbb F_p)\equiv F_r(\l)_{{p-1}}\equiv F_r(\l)_{\frac{p-1}2} \pmod p.$$  
In particular, when $r=3$, $\mathcal{X}_3(\l)$ is 
a family of K3 surfaces that has been studied in \cite{AOP, Long}, denoted by $S_\l = \mathcal{X}_3(\l)$.  Over an arbitrary finite field $\mathbb F$ 
containing $\l$, Ahlgren, Ono, and Penniston showed that $\#(S_\l/\mathbb F)$  can be computed using points on $E_\l: y^2=(x-1)(x^2-\frac 1{1-\l})$ 
over $\mathbb F$  \cite{AOP}.  By Dwork \cite{Dwork}, for any $\l\in \Z_p$ such that $F_r(\l)_{p-1}\neq 0 \pmod p$ (i.e. $p$ 
 \emph{ordinary}  for $\mathcal X_r(\l)$) and for any integer $m\ge 1$,
\begin{equation}\label{Dworkslimit}  
F_r(\l)_{mp^s-1} \equiv \gamma(\l){F_r(\l^p)_{mp^{s-1}-1}} \pmod{p^s} \end{equation} 
for a $p$-adic unit $\gamma(\l)$ which is independent of $m$, but may vary if $\l $ is replaced by $\l^p$.

\begin{theorem}\label{prop1}Let $p$ be an odd prime, $\l\in \Z_p$ such that $\mathcal X_r(\l)$ has  good ordinary reduction at $p$.  Then 
$\ell(\tau)=\sum_{n\ge 1} \frac{F_r(\l)_n}{2n+1}\tau^{2n+1}$ is the logarithm of a formal group over $\Z_p$, which is isomorphic to a 
formal group attached to $\mathcal X_r(\l)$ constructed by Stienstra \cite{Stienstra}. When $r=3$, 
 for all integers $s\ge 1$ and $m$ odd
$${F_3(\l) _{\frac{mp^s-1}2}}\equiv  \left( \frac{\l-1}{p} \right)\alpha_{p,\lambda}{}^2 {F_3(\l)_{\frac{mp^{s-1}-1}2}} \pmod {p^s},$$ 
with $\alpha_{p,\l}$  being  the unit root of
$X^2-[p+1-\#(E_\l/{\mathbb F}_p)]X+p=0$. 
\end{theorem}

At special values of $\l$ such that $E_\l$ has complex multiplications (CM), stronger congruences have been observed. These congruences are 
known as \emph{supercongruences}.  Rodriguez-Villegas
conjectured several supercongruences involving truncated hypergeometric series in \cite{RV}, including the following:  for odd primes $p$,
$$F_3(1)_{\frac{p-1}{2}} = {}_{3}F_{2}\left(\begin{array}{c}
\frac{1}{2},\frac{1}{2},\frac{1}{2}\\
1,1
\end{array};1\right)_{\frac{p-1}2}\equiv b_p \pmod {p^2}$$  where $b_p$ is the $p$th coefficient of the weight 3 cusp 
form $\eta(4z)^6$, where $\eta(z)=q^{1/24} \prod_{n=1}^\infty (1-q^n)$ with $q=e^{2\pi i z}$, is the eta function.  
The $\l=1$ case was proved by Van Hamme in \cite[1996]{VH} and by Ono in \cite[1998]{Ono2}, using different methods.  

Similarly, Z.-W. Sun conjectured (see remark 1.4 in \cite{ZWSun3}) a congruence for the $\lambda=64$ case: 
$$_3F_2\left (\begin{array}{c}
\frac{1}{2},\frac{1}{2},\frac{1}{2}\\
1,1
\end{array};64 \right )_{\frac{p-1}2}=
\sum_{k=0}^{\frac{p-1}{2}} \left(\frac{(\frac 12)_k}{k!}\right)^3(64)^k\equiv a_p \pmod {p^2}$$ where $a_p=0$ if $p\equiv 3,5,6 \mod 7$ and
$a_p=4x^2-2p$ where $p=x^2+7y^2, x,y\in \mathbb Z$, if $p \equiv 1,2,4 \pmod{7}$.
In fact, this $a_p$ is just the $p$th coefficient of $\eta(z)^3\eta(7z)^3$.

\begin{theorem}\label{maintheorem}
Let $\lambda\neq 1$ be an algebraic number such that $E_\lambda$ has
complex multiplications. Let $p$ be a prime and $E_\l$   have a model defined over ${\mathbb Z}_p$ with   good
reduction modulo $p{\mathbb Z}_p$. Then
$$_3F_2\left ( \begin{array}{cccc}1/2,&1/2,&1/2  \\ &1&1\end{array}; \l \right )_{\frac{p-1}2}=\sum_{k=0}^{\frac{p-1}{2}}
\left(\frac{(\frac 12)_k}{k!}\right)^3\l^k\equiv \left( \frac{\l-1}{p} \right)\alpha_{p,\lambda}{}^2 \pmod {p^2}$$
where $\alpha_{p,\lambda}$ is the unit root of
$X^2-[p+1-\#(E_\l/{\mathbb F}_p)]X+p=0$ if  $E_\l$ is ordinary at $p$; and
$\alpha_{p,\lambda}=0$ if $E_\lambda$ is supersingular at $p$.
\end{theorem}
Our result confirms the Conjecture of Sun mentioned above.  
We conjecture that the statement of Theorem \ref{prop1} is true modulo $p^{2s}$ when $E_\l$ has CM.

The hypergeometric series $_3F_2\left (\begin{array}{c}
\frac{1}{2},\frac{1}{2},\frac{1}{2}\\
1,1 \end{array}; \l \right )$, when \emph{not} 
truncated, gives an expression for the real period of the elliptic curve $E_\l$ as McCarthy shows in \cite{McCarthy}.

We derive the following corollary to Theorem \ref{maintheorem} in section \S \ref{corollarysec}:
 \begin{corollary} \label{maincorollary} Let  $H_k$ be the harmonic sum $\sum_{j=1}^k \frac 1j$.
If $E_\l$ is a CM elliptic curve,   then for almost all primes $p$ such that $\lambda$ embeds in $\mathbb{Z}_p$,
\[ \sum_{i=0}^{\frac{p-1}{2}} \binom{2i}{i}^3 \left( \frac{\lambda}{64} \right)^i
\left( 6(H_{2i} - H_i) + \left( \frac{\left( \frac{\lambda}{64} \right)^{p - 1} - 1}{p} \right) \right) \equiv 0 \pmod p. \]
\end{corollary}

Below is one simple, special case of these congruences for $\lambda=64$.

\begin{corollary}
For  all primes $p>3$, we have
 \begin{equation}\label{cor:4}
   \sum_{i=1}^{\frac{p-1}2} \binom{2i}{i}^3 \sum_{j=1}^{i} \frac{1}{i+j} \equiv 0 \pmod{p}.
 \end{equation}
\end{corollary}

In general, such  congruences  are difficult to prove. For similar work, see  \cite{Ahlgren,AO}.  
Remark 1 of \cite{Long2} reduces an open supercongruence  to a 
 congruence like \eqref{cor:4}.

We end our introduction with another motivation for supercongruences.
It is known that the coefficients of  weight-$k$ noncongruence modular forms satisfy the so-called Atkin and Swinnerton-Dyer congruences \cite{ASD, Scholl}. 
These congruences are  supercongruences if  $k>2$ \cite{Scholl} and have played an important role in understanding the characterizations of genuine noncongruence 
modular forms \cite{LL}.

\section{Atkin and Swinnerton-Dyer congruences of a family of truncated hypergeometric series}

The Hasse invariants of the Legendre family of elliptic curves     $L_\l: y^2=x(x-1)(x-\l)$) are  $A_p(\l)=(-1)^{(p-1)/2} \sum_{i=0}^{(p-1)/2} \binom{\frac{p-1}2}{i}^2\l^i$ and
$$A_p(\l)\equiv [p+1-\#L_\l(\mathbb F_p)]\equiv (-1)^{(p-1)/2}  \,  _2F_1 \left(\begin{array}{c}
\frac{1}{2},\frac{1}{2}\\
1
\end{array};\l\right)_{(p-1)/2}\pmod p$$
where  $  _2F_1 \left(\begin{array}{c}
\frac{1}{2},\frac{1}{2}\\
1
\end{array};\l\right)$ is the unique-up-to-scalar, holomorphic-near-0 solution of the Picard-Fuchs equation of $L_\l$ (see \cite{Clemens}).    
In \cite{Dwork}, Dwork proved that when  $A_p(\l)\not \equiv 0 \pmod p$,  i.e. when $L_\l$ is ordinary at $p$,
$$ \frac{{ _{2}F_{1}\left(\begin{array}{c}
\frac{1}{2},\frac{1}{2}\\
1
\end{array};  \lambda\right)_{p^s-1}}}{{_{2}F_{1}\left(\begin{array}{c}
 \frac{1}{2},\frac{1}{2}\\
1
\end{array}; \lambda^p\right)_{p^{s-1}-1}}}   \equiv    \gamma_p(\l)  \pmod {p^s} $$ where $\gamma_p(\l)$ is a $p$-adic unit  
which usually varies when $\l$ is replaced by $\l^p$. 
If $\l=\l^p$, the limit is
  $ \left ( \frac{-1}p\right )\beta_p$, where  $\beta_p$ is the $p$-adic unit root of $T^2-[p+1-\#(L_\l/\mathbb F_p)]T+p=0$.
When $L_\l$ has CM, the $\gamma_p$ can often be obtained via Gauss sums and Jacobi sums.
 Yu has further extended Dwork's results to Dwork families of algebraic varieties
 \cite{JDYu}.

We now  compare Dwork's result with what can be predicted from Atkin and Swinnerton-Dyer congruences. Let $P_n(x)$ denote the $n$th Legendre polynomial, which can be defined by
$P_n(x) = \frac{1}{2^n n!} \frac{d^n}{dx^n} (x^2-1)^n$ \cite{andrews, CosterVanHamme, Sun}.
These polynomials form an important class of orthogonal polynomials and  have several nice properties; but
the fact most relevant to our application is that they have generating function $(1-2xt+t^2)^{-1/2}= \sum_{n=0}^\infty P_n(x)t^n$.
Because of this, special values of $P_n(x)$ show up in certain expansions of differential forms on elliptic curves.
The first few Legendre Polynomials are $P_0(x)=1$, $P_1(x)=x$, and $P_2(x)=\frac{1}{2}(3x^2-1)$.

For elliptic curves of the form $\mathcal E: y^2=x(x^2+Ax+B)$ defined over $\mathbb Z_p$ with $t=x/y$ as a local parameter at the point at infinity
(where $t$ has a simple zero), Coster and van Hamme showed that the coefficients of the $t$-expansion of the invariant differential form
$-\frac{dx}{2y}$ of $\mathcal{E}$ are essentially just special values of Legendre Polynomials (see formula (1) of \cite{CosterVanHamme}).
In particular, for $L_\l: y^2=x(x-1)(x-\l)$,
\begin{equation}\label{omegaexpansion}
 -\frac{dx}{2y} = \sum_{k=0}^\infty P_k\left( \tfrac{1+\lambda}{1-\lambda}\right) (\lambda-1)^k t^{2k+1} \frac{dt}{t}.
 \end{equation}
For $k\geq 0$, letting $a_{2k+1} = P_k\left( \tfrac{1+\lambda}{1-\lambda}\right)(\lambda-1)^k$ and using  formulas (5) and (6) in \cite{Zagier},  we have 
\begin{equation}\label{a_k} a_{2k+1} = {}_2F_1(\begin{smallmatrix}-k, \; -k, \\ 1 \end{smallmatrix}; \lambda) (-1)^k=
{}_2F_1(\begin{smallmatrix}-k, \;1+k \\ 1 \end{smallmatrix}; \tfrac{-\l}{1-\l}) (\lambda-1)^k.
\end{equation}
Note that these are terminating hypergeometric series; i.e. degree $k$ polynomials of $\l$,  because of the $-k$ argument.

The Atkin and Swinnerton-Dyer congruences (ASD) for elliptic curves (Theorem 4 of \cite{ASD}) imply that if $\l$ embeds in
$\mathbb{Q}_p$ and $L_\l$ has good reduction modulo $p$, then for all positive integers $m,s$,
\begin{equation} \label{ASDcong} a_{mp^{s+1}}- A_p  a_{mp^s}+p  a_{mp^{s-1}}  \equiv 0 \pmod{p^{s+1}}\end{equation}
where $A_p = p+1-\#(L_\l/{\mathbb F}_p)$.
We define $a_k$ to be 0 if $k$ is not integral, as may happen for the final term if $s=0$.
The factors of $(-1)^k$ and $(\l-1)^k$ can be omitted from the expressions \eqref{a_k} for $a_k$ if we adjust the middle coefficient
of the ASD congruence by the Legendre symbols $\left(\frac{-1}{p}\right)$ or $\left(\frac{\l-1}{p}\right)$,
respectively.

Essentially, the ASD congruences say that for fixed $p$ and $m$, terms of the sequence $\{a_{mp^s}\}$ satisfy
a three-term congruence with increasing $p$-adic precision as $s$ increases.  The ASD
congruences generalize the Hecke recursion: Fourier coefficients $a_n$ of weight $k=2$, normalized Hecke newforms with
trivial nebentypus satisfy the three-term recursion, for all $m,s \geq 1$ and all $p$,
\begin{equation}\label{Hecke} a_{mp^{s+1}} - a_p a_{mp^s} + p a_{mp^{s-1}}=0. \end{equation} 

In the ASD congruences for an elliptic curve $\mathcal{E}$, we distinguish two cases. If the middle coefficient $A_p$ is divisible by $p$, 
we say that $\mathcal{E}$ is
\emph{supersingular} at $p$ or simply that $p$ is supersingular.  Otherwise, we say $\mathcal{E}$ is \emph{ordinary} at $p$ or that $p$ is ordinary.
Dwork's congruences, in which consecutive ratios of certain terms in a sequence converge to a $p$-adic limit, are related to
ASD congruences at ordinary primes. At ordinary primes, let $\beta_{p,\l}$ be
the $p$-adic unit root of $T^2-[p+1-\#(L_\l/{\mathbb F}_p)]T+p$.  Then the ASD congruences imply that
$\displaystyle \frac{a_{p^s}}{a_{p^{s-1}}} \equiv \beta_{p,\l} \pmod{p^s}.  $
Note that $a_1=1$, so the $s=1$ case of this congruence is just $a_p \equiv \beta_{p,\l} \pmod{p}$.

Thus, at all ordinary primes $p$ of $L_\l$, we obtain congruences for hypergeometric functions using the expansions 
given in \eqref{a_k} with $k=\frac{p^s-1}{2}$:

\begin{equation}\label{2F1cong}
\frac{
_2F_1\left(\begin{smallmatrix}\tfrac{1-p^s}{2}, \; \tfrac{1\pm  p^s}{2} \\ 1 \end{smallmatrix}; \l \right)
}{
_2F_1\left(\begin{smallmatrix}\tfrac{1-p^{s-1}}{2}, \; \tfrac{1\pm p^{s-1}}{2} \\ 1 \end{smallmatrix}; \l \right)
} \equiv \left(\frac{-1}{p}\right) \beta_{p,\l} \pmod{p^s}.
\end{equation}

The twist by the character $\left(\frac{-1}{p} \right)$ accounts for the factor $(-1)^k$ in the first equality in \eqref{a_k} and 
for the  change of argument from $\frac{-\l}{1-\l}$ to $\l$ and the factor $(\l-1)^k$ in the second equality in \eqref{a_k}. \footnote{The curve
$L_{\frac{-\l}{1-\l}}$ is isomorphic to the twist of $L_\l$ by $\left(\frac{1-\lambda}{p} \right)$.  Combining this with the 
factor of $(\lambda -1)^k$, we get a twist by $\left( \frac{-1}{p} \right)$.}


Both the Hecke recursion and ASD congruences are related to formal groups.   Any sequence of $p$-adic integers 
$a_n$ with $a_1$ being a $p$-adic unit satisfying the
congruences \eqref{ASDcong} can be used to construct a formal group law  $F(x,y):=\ell^{-1}(\ell(x)+\ell(y)) \in \mathbb{Z}_p[[x,y]]$ with formal logarithm $\ell(x)=\sum_{n=1}^\infty
\frac{a_n}{n}x^n$.  Note that the power series $\ell(x)$ and $\ell^{-1}(x)$  
have denominators with arbitrarily large powers of $p$ in general; it is the congruences \eqref{ASDcong} that guarantee that the 
composition $F(x,y)$ has no denominators of $p$. 
Details can be found in \cite{Dit1} or other formal group theory references; specifically, 
the $p$-typification of $\ell^{-1}(\ell(x)+\ell(y))$ is isomorphic over $\mathbb{Z}_p$ to a formal group law with $F_p$-type 
$F_p = \{A_p\}- V_p$, where $F_p$ is the $p$th Frobenius operator, $\{A_p\}$ is a Hilbert operator, and $V_p$ is the
$p$th Verschiebung operator in the Cartier-Dieudonn\'e module.   (This integer $A_p$ corresponds to the eigenvalue of 
the $p$th Hecke operator; recall that the Hecke operator is just the sum of the Frobenius and Verschiebung operators on the space of congruence cusp forms.) 

 In fact, in the ordinary case, when $a_p \not \equiv 0 \pmod{p}$, our formal group is isomorphic 
 to one with $F_p$-type $F_p = \{ \beta_p\}$, where $\beta_p$ is a 
 $p$-adic unit.  This fact corresponds to the congruences $a_{mp^s} \equiv \beta_p 
 a_{mp^{s-1}} \pmod{p^s}$, which account for the many congruences we consider 
 for expressions of the form $\frac{a_{p^s}}{a_{p^{s-1}}}$.   Dwork's congruences, 
 in which the denominator involves $\l^p$ instead of $\l$, correspond to a 
 formal group law over $\mathbb{Z}_p[\l]$, which is isomorphic over $\mathbb{Z}_p[[\l]]$ 
 to a formal group law with $F_p$-type $F_p = \{ \gamma_p(\l) \}$, 
 for some $\gamma_p(\l) \in \mathbb{Z}_p[[\l]]$.  Dwork's formal group law can be 
 specialized to many different formal group laws over $\mathbb{Z}_p$ by choosing 
 suitable $\l \in \mathbb{Z}_p$ (as we show in Proposition \ref{Dworkprop}), but the $F_p$-type cannot be specialized by 
 substituting the value of $\l$ in $\gamma_p(\l)$, because the Hilbert structures on 
 $\mathbb{Z}_p[[\l]]$ and on $\mathbb{Z}_p$ are incompatible.\footnote{Dwork's 
 congruences make use of the endomorphism $\tau$ of $\mathbb{Z}_p[[\l]]$ 
 sending $f(\l)$ to $f(\l^p)$.  This endomorphism satisfies $\tau(f(\l))=f(\l^p) \equiv 
 f(\l)^p \pmod{p}$ for all $f(\l) \in \mathbb{Z}_p[\l]$ and is a ring endomorphism,  
 so it gives a Hilbert structure to formal group laws over $\mathbb{Z}_p[[\l]]$.  
 On $\mathbb{Z}_p$, however, we must use the identity endomorphism $\iota$, which also satisfies  
 $\iota(x)=x \equiv x^p \pmod{p}$ for all $x \in \mathbb{Z}_p$.  This is not compatible with a specialization of $\tau$ 
 unless we choose $\l$ satisfying $\l = \l^p$.}  So, even though the many specializations of Dwork's formal 
 groups are all isomorphic over $\mathbb{Z}_p$ as long as $\l \pmod{p}$ is fixed, 
 Dwork's congruenes give different limits $\gamma_p(\l)$ for these $\l$.  
 When we omit the $p$th power from the denominator, we obtain the same limit $\bigl( \frac{-1}{p} \bigr) \beta_p$ for 
 all $\l$ with $\l \pmod{p}$ fixed.  Note that $\gamma_p(\l)= \bigl( \frac{-1}{p} \bigr) \beta_p$ if $\l=\l^p$.

The perspective of formal groups motivates our approach to congruences and supercongruences of hypergeometric functions; many of the congruences
found in the literature seem to be initial cases of the ASD congruence structure.  In fact, our second theorem appears to be the very first
case of an ASD congruence; in Conjecture \ref{maintheoremconj}, we suggest that infinitely many more congruences hold.

Recall $F_r(\l)_n:=\, _rF_{r-1}\left ( \begin{array}{ccccc}1/2,&1/2,&\cdots,&1/2 \\ &1,& \cdots,&1\end{array};\l \right )_{n}$.
\begin{prop} \label{Dworkprop}
Let $p$ be an odd prime and  $A(n)$ be a sequence of numbers in $\Z_p$ such that $A(0)$ is a unit, and suppose the following three 
conditions hold. 

a) For all $n,m,s \in \mathbb{Z}_+$, $$ \displaystyle \frac{A(n+mp^{s+1})}{A([\frac n p]+mp^s)}\equiv \frac{A(n)}{A([\frac n p])}\pmod{p^{s+1}}.$$

b) For all $n \in \mathbb{Z}_+$, $A(n)/A([\frac np ])\in \Z_p$.

c) $A(i)\equiv 0 \pmod p$ for all $\frac{p+1}2\le i<p$.

Then for any $x\in \Z_p$ such that $\sum_{i=0}^{p-1} A(i)x^i\not\equiv 0\pmod p$, there exists a $p$-adic unit $\alpha$, 
depending only on $x \pmod p$, such that for any odd integer $m$ 
$$ 
\frac{ \sum_{i=0}^{\frac{mp^{s}-1}2}A(i)x^{i}}{ \sum_{i=0}^{\frac{mp^{s-1}-1}2}A(i)x^{i}} \equiv \alpha \mod p^{s-d_m}$$
where $d_m = {\rm{max}}_{s\geq 0} \left( {\rm{ord}}_p(\sum_{i=0}^{\frac{mp^{s}-1}2}A(i)x^{i}) \right)$.  
\end{prop}

The proof implies that $d_m$ is finite, unless $\sum_{i=0}^{\frac{mp^{s}-1}2}A(i)x^{i}=0$ for all $s$.  Also, if 
$\sum_{i=0}^{\frac{mp-1}{2}} A(i)x^i\not\equiv 0\pmod p$, then $d_m=0$; note that $d_1=0$.

\begin{proof} Compare to Theorem 2 of Dwork; we assume $A(n)=B(n)$ (see Dwork's notation) and we add condition c), 
which will allow us to go use sums to $\frac{mp^s-1}2$ instead of to $mp^s-1$. Let $F(x)=\sum_{i\ge 0} A(i)x^i.$ 
Taking the sum of congruence (2.1) in \cite{Dwork} from $m=0$ to $m=n$, we have 
\[F(x)\sum_{j=0}^{(n+1)p^s-1} A(j)x^{pj}\equiv 
F(x^p)\sum_{j=0}^{(n+1)p^{s+1}-1} A(j)x^{j} \mod p^{s+1}\mathbb{Z}_p[[x]].\]

Under the additional assumption c) and the argument of Dwork (Theorem 2), we can obtain that for all integers $n,s\ge 0$
\[F(x)\sum_{j=np^s+\frac{p^s+1}2}^{(n+1)p^s-1} A(j)x^{pj}\equiv 
F(x^p)\sum_{j=np^s+\frac{p^{s+1}+1}2}^{(n+1)p^{s+1}-1} A(j)x^{j} \mod p^{s+1}\mathbb{Z}_p[[x]].\]
(For full details and notation, please see \cite{Dwork}.)
If we subtract the congruences above, we obtain
\begin{equation} \label{Dwork2.1}
F(x)\sum_{j=0}^{\frac{mp^s-1}{2}} A(j)x^{pj}\equiv 
F(x^p)\sum_{j=0}^{\frac{mp^{s+1}-1}{2}} A(j)x^{j} \mod p^{s+1}\mathbb{Z}_p[[x]]
\end{equation}
where $m=2n+1$ could be any odd number.

For fixed $m$, we divide congruence \eqref{Dwork2.1} for consecutive $s$, obtaining 
\begin{equation} \label{ratiocong}
  \frac{\sum_{i=0}^{\frac{mp^{s+1}-1}2}A(i)x^i}{\sum_{i=0}^{\frac{mp^{s}-1}2}A(i)x^{i}}\equiv 
  \frac{\sum_{i=0}^{\frac{mp^{s}-1}2}A(i)x^{ip}}{\sum_{i=0}^{\frac{mp^{s-1}-1}2}A(i)x^{ip}} \mod p^{s-d_m}\mathbb{Z}_p[[x]].
\end{equation}

Consequently, the left-hand side, when viewed as a formal power series of the form $\sum_{n\ge 0}  a_{n,s+1}x^n$,
satisfies that $p^{s-d_m}\mid a_{n,s+1}$ if $p\nmid n$.  Iterating this idea, we have 
$$\frac{\sum_{i=0}^{\frac{mp^{s+1}-1}2}A(i)x^i}{\sum_{i=0}^{\frac{mp^{s}-1}2}A(i)x^{i}}\equiv 
\frac{\sum_{i=0}^{\frac{mp^{s}-1}2}A(i)x^{ip}}{\sum_{i=0}^{\frac{mp^{s-1}-1}2}A(i)x^{ip}}\equiv  
\frac{\sum_{i=0}^{\frac{mp^{s-1}-1}2}A(i)x^{ip^2}}{\sum_{i=0}^{\frac{mp^{s-2}-1}2}A(i)x^{ip^2}} \mod p^{s-1-d_m}\mathbb{Z}_p[[x]].$$ 

So $p^{s-1-d_m} \mid a_{n,s+1}$ if $p^2 \nmid n$. By induction, $p^{s-i-d_m}\mid a_{n,s+1}$ if $p^{1+i}\nmid n$.  

Using this, we show that the left hand side of \eqref{ratiocong}, for fixed $m$ and $s$, is determined modulo $p^{s+1-d_m}$ by $\l \pmod{p}$, provided that 
$\sum_{i=0}^{p-1}A(i)x^i\not \equiv 0 \pmod p$.   Note that condition b) then implies  that 
$\sum_{i=0}^{p^s-1}A(i)x^i$ is a $p$-adic unit for each $s$.  We show that 
$$\frac{\sum_{i=0}^{\frac{mp^{s+1}-1}2}A(i)x^{i}}{\sum_{i=0}^{\frac{mp^{s}-1}2}A(i)x^{i}}- 
\frac{\sum_{i=0}^{\frac{mp^{s+1}-1}2}A(i)(x+pk)^i}{\sum_{i=0}^{\frac{mp^{s}-1}2}A(i)(x+pk)^i}= 
\sum_{n\geq 0} a_{n,s+1}(x^n-(x+pk)^{n})$$ is congruent to 0 modulo $p^{s+1-d_m}$ for any $k \in \mathbb{Z}_p$.  
Writing $n=p^en'$ with $p\nmid n'$, then  $p^{s-e-d_m}\mid a_{n,s+1}$ 
by the paragraph above, and $x^n-(x+pk)^n \equiv 0 \mod p^{e+1}$ so $p^{s+1-d_m} \mid a_{n,s+1}(x^n-(x+pk)^{n})$ term by term.

Since $x \equiv x^p \pmod{p}$, for fixed $x$ such that $\sum_{i=0}^{p-1}A(i)x^i\not \equiv 0 \pmod p$ 
we can replace congruence (\ref{ratiocong}) with 
\begin{equation} 
  \frac{\sum_{i=0}^{\frac{mp^{s+1}-1}2}A(i)x^i}{\sum_{i=0}^{\frac{mp^{s}-1}2}A(i)x^{i}}\equiv 
  \frac{\sum_{i=0}^{\frac{mp^{s}-1}2}A(i)x^{i}}{\sum_{i=0}^{\frac{mp^{s-1}-1}2}A(i)x^{i}} \mod p^{s-d_m};
\end{equation}
thus there is some $\alpha \in \mathbb{Z}_p$ such that 
$\dfrac{\sum_{i=0}^{\frac{mp^{s}-1}2}A(i)x^i}{\sum_{i=0}^{\frac{mp^{s-1}-1}2}A(i)x^{i}} \equiv \alpha \pmod{p^{s-d_m}}$.

To see that this limit $\alpha$ is independent of $m$ as long as $x \pmod{p}$ remains fixed, we rearrange
congruence \eqref{Dwork2.1} to obtain, for any odd $m$, 
\[ \frac{F(x)}{F(x^p)} \equiv \dfrac{\sum_{i=0}^{\frac{mp^{s}-1}2}A(i) x^i}{\sum_{i=0}^{\frac{mp^{s-1}-1}2}A(i) x^{ip}}
\equiv \dfrac{\sum_{i=0}^{\frac{p^{s}-1}2}A(i) x^i}{\sum_{i=0}^{\frac{p^{s-1}-1}2}A(i) x^{ip}} \pmod{p^{s-d_m}\mathbb{Z}_p[[x]]}. \]

We may choose $\hat x$ to be the Teichmuller lift of $x$, so that $\hat x \equiv x \pmod{p}$ and $\hat x^p = \hat x$.  Then 
$\dfrac{\sum_{i=0}^{\frac{mp^{s}-1}2}A(i) \hat x^i}{\sum_{i=0}^{\frac{mp^{s-1}-1}2}A(i) \hat x^{i}}
\equiv \dfrac{\sum_{i=0}^{\frac{p^{s}-1}2}A(i) \hat x^i}{\sum_{i=0}^{\frac{p^{s-1}-1}2}A(i) \hat x^{i}} \equiv \alpha \pmod{p^{s-d_m}}$.
Thus, the limit $\alpha$ is independent of $m$.
\end{proof}

Now take $A(n)=\left ( \frac{(\frac{1}{2})_n}{n!} \right )^r$ for any integer $r\ge 1$. It follows from \cite{Dwork},  
a), b) and c) of the above  hold. Then $\sum_{j=0}^n A(j)x^j=F_r(\l)_n$. 
\begin{theorem}\label{maintheorem2}For any positive integer $r$, odd prime $p$, and $\l\in \Z_p$ such that 
$F_r(\l)_{\frac{p-1}{2}} \not \equiv 0 \pmod p$, there is a $p$-adic unit $\alpha_{r}$, such that for all integers $s\ge 1$ and $m$ odd
\[
 \frac{F_r(\l)_{\frac{mp^s-1}2}}{F_r(\l)_{\frac{mp^{s-1}-1}2}}\equiv\alpha_r \pmod {p^{s-d_m}}.
\] 
Moreover, the formal group with logarithm $\sum_{n\ge 0} \,  \frac{F_r(\l)_n}{2n+1}\tau^{2n+1}$ is isomorphic over $\mathbb{Z}_p$ to 
the formal group $H^{r-1}(\mathcal X_r,\hat G_{m,\mathcal X_r})$ where $\mathcal X_r(\l): W^2=X_1X_2\cdots X_r(X_1-X_2)(X_2-X_3)\cdots(X_r-\l X_1)$.
\end{theorem}

\begin{proof}
  From the previous proposition, we know the first claim holds. Let $\ell(\tau)=\sum_{n\ge 0} \,  \frac{F_r(\l)_n}{2n+1}\tau^{2n+1}$. 
  It follows that the formal group law $\ell^{-1}(\ell(x)+\ell(y))$ is integral at $p$.

By Theorem 2 of \cite{Stienstra}, there is a 1-dimensional formal group  $H^{r-1}(\mathcal X_r,\hat G_{m,\mathcal X_r})$ with 
logarithm $\sum \frac{b_{r, n}(\l)}{2n+1} \tau^{2n+1}$ where 
  $b_{r,n}(\l)=\, _rF_{r-1}\left ( \begin{array}{ccccc}-n,&\cdots,&-n &  \\ 1,& \cdots,&1\end{array}; (-1)^r\l \right ).$ 
  By an observation of Koblitz \cite{Koblitz}, we know 
  $$\lim_{s\rightarrow \infty}\frac{b_{r,\hat \l, (mp^s-1)/2}}{b_{r, \hat \l, (mp^{s-1}-1)/2}}=
  \lim_{s\rightarrow \infty}\frac{a_{r,\hat \l, (mp^s-1)/2}}{a_{r, \hat \l, (mp^{s-1}-1)/2}}=\alpha_{r} \in \mathbb{Z}_p,$$ 
  which implies that both formal groups are isomorphic over $\mathbb{Z}_p$ to a formal group law with $F_p$-type $F_p = \{\alpha_r\}$, 
  and thus are isomorphic to each other over $\mathbb{Z}_p$. 
\end{proof}
(When $r=3,\l=1$, these formal groups are also isomorphic to a formal Brauer group constructed in \cite{SB}. )

In fact, we expect this formal group law to be integral at every prime except $2$ and all prime 
places $\mathfrak{p}$ such that $\rm{val}_\mathfrak{p}(\l) <0$, but the integrality 
at all ordinary and inert primes follows immediately from Theorem \ref{maintheorem2}.


\section{Supercongruences}
Any formal group law with logarithm $\ell(x)= \sum_{n\geq 1} \frac{a_n}{n}x^n$ that is 
integral at $p$ and that has finite $F_p$-type satisfies an infinite family of 
congruences, which express $a_{mp^s}$ as a linear combination of lower-index coefficients 
$a_n$ modulo $p^s$.  Many of the congruences for hypergeometric series 
can be seen from this perspective; however, the congruences are often much stronger, 
giving a formula for $a_{mp^s}$ modulo $p^{2s}$, $p^{3s}$, or even $p^{4s}$ or more.  
Such congruences, that are stronger than what are predicted by the existence of a 
formal group, are called \emph{supercongruences}.  One source of supercongruences are ASD congruences 
of Fourier coefficients of cusp forms with weight $k>2$; these exhibit congruences of order $p^{(k-1)s}$ \cite{Scholl}.  
We are interested in another source of supercongruences: extra symmetries of the underlying variety, such as 
complex multiplications of the elliptic curves. 

 Here is a well-known example of a supercongruence.  Beukers conjectured that for all odd primes $p$
$$_{4}F_{3}\left(\begin{array}{c}
\frac{1-p}{2},\frac{1-p}{2},\frac{1+p}{2},\frac{1+p}{2}\\
1,1,1
\end{array};1\right)\equiv c_p\pmod {p^2}$$ where the left hand side is the $\frac{p-1}2$th Ap\'ery number $\sum_{k=0}^{(p-1)/2}\binom{(p-1)/2}{k}^2\binom{(p-1)/2+k}{k}^2$ and $c_p$ is the $p$th coefficient of  the weight-4 modular
form $\eta(2z)^4\eta(4z)^4$; this was proved by Alhgren and Ono \cite{AO}. Ahlgren and Ono's approach uses Gaussian hypergeometric functions
(see \cite{AO} and \cite[Chapter 11]{Ono}) and has inspired much later work including Kilbourn's result (\cite{Kilbourn}) that for all primes $p>2$
$$_{4}F_{3}\left(\begin{array}{c}
\frac{1}{2},\frac{1}{2},\frac{1}{2},\frac{1}{2}\\
1,1,1
\end{array};1\right)_{(p-1)/2}\equiv c_p\pmod {p^3}.$$ 

We would like to understand when we should expect supercongruence for our truncated hypergeometric series and have a better understanding for $r=2$ and $3$ cases by using the following theorem of Coster and van Hamme.

\begin{theorem}[Coster and van Hamme, \cite{CosterVanHamme}]\label{CosterVanHammeTheorem}
Let $p$ be an odd prime and $d$ a square-free positive integer such that
$(\frac{-d}{p})=1$.  Let $K$ be an algebraic number field such that $\sqrt{-d}
\in K$ and $K \subset \mathbb{Q}_p$.  Consider the elliptic curve
\[\mathcal{E}: Y^2=X(X^2+AX+B)\] with $A,B \in K$, where $A$ and $\Delta=A^2-4B$ are
$p$-adic units.  Let $\omega$ and $\omega'$ be a basis of periods of $\mathcal{E}$
and suppose that $\tau = \omega'/\omega \in \mathbb{Q}(\sqrt{-d})$, $\tau$ has positive imaginary
part.  Let $\pi, \bar \pi \in \mathbb{Q}(\sqrt{-d})$ be complex conjugates such that
$\pi \bar \pi = p$, with $\bar \pi$ a $p$-adic unit, $\pi=u_1 +v_1 \tau$, and
$\pi \tau = u_2 + v_2 \tau$ with $u_1, v_1, u_2, v_2$ integers and $v_1$ even.
Then we have
\begin{equation}
P_{\frac{mp^r-1}{2}}\left(\frac{A}{\sqrt{\Delta}}\right) \equiv
\varepsilon^{mp^{r-1}} \cdot \bar \pi \cdot
P_{\frac{mp^{r-1}-1}{2}}\left(\frac{A}{\sqrt{\Delta}}\right) \pmod{\pi^{2r}},
\end{equation}
where $m$ and $r$ are positive integers, with $m$ odd, and
$\varepsilon = i^{-u_2v_2 + v_2+p-2}$, where $P_n(x)$ is the $n$th Legendre polynomial.
\end{theorem}Note that the assumptions
imply that $\mathcal{E}$ has CM,  $A=3\wp(\frac{1}{2}\omega)$, $\sqrt{\Delta}=\wp(\frac{1}{2}\omega'
+ \frac{1}{2} \omega) - \wp(\frac{1}{2}\omega')$, where $\wp(z)$ is the Weierstrass
$\wp$-function. The technical details such as $A$ and $D$ being $p$-adically integral are always satisfied in our cases, and
$\varepsilon^{mp^{r-1}}$ is just a quartic character, which we can explicitly identify.  The main point of the theorem is the
existence of supercongruences  arising from an elliptic curve $\mathcal E$ with CM.
While Coster and van Hamme interpreted the congruence as inclusion in an ideal of the ring
of integers of $K$, we interpret all of our congruences $p$-adically and simply replace $\pmod{\pi^{2r}}$ with $\pmod{p^{2r}}$.

Thus, for CM curves $L_\l$, we can double the strength of the congruences (\ref{2F1cong}).  The factor $(\lambda - 1)^k$ from the 
formulas \eqref{a_k} and \eqref{omegaexpansion} only shows up as a quadratic character in the congruences \eqref{2F1cong}, because 
$(\l-1)^{\frac{mp^s-mp^{s-1}}{2}} \equiv \left( \frac{\l-1}{p} \right) \pmod{p^s}$; however for supercongruences modulo $p^{2s}$, 
it must be included.  If $\l$ is a CM value for the family $E_\l$, then at ordinary primes
 $p$, for all  $m, s \geq 1$ with $m$ odd, we have 
\begin{align*}
& {
_2F_1\left(\begin{smallmatrix}\tfrac{1-mp^s}{2}, \; \tfrac{1 - m p^s}{2} \\ 1 \end{smallmatrix}; \l \right)
} \\ & \equiv \left( \frac{1 - \l }{p} \right)  \left( \l-1 \right)^{\frac{m p^s-mp^{s-1}}{2}} \beta_{p,\l} \; {
_2F_1\left(\begin{smallmatrix}\tfrac{1-mp^{s-1}}{2}, \; \tfrac{1 - mp^{s-1}}{2} \\ 1 \end{smallmatrix}; \l \right)
} \pmod{p^{2s}}
\end{align*}  and 
\begin{equation*}
{
_2F_1\left(\begin{smallmatrix}\tfrac{1-mp^s}{2}, \; \tfrac{1 + m p^s}{2} \\ 1 \end{smallmatrix}; \l \right)
} \equiv  \left( \frac{-1}{p} \right) \beta_{p,\l} \; {
_2F_1\left(\begin{smallmatrix}\tfrac{1-mp^{s-1}}{2}, \; \tfrac{1 + mp^{s-1}}{2} \\ 1 \end{smallmatrix}; \l \right)
} \pmod{p^{2s}}.
\end{equation*}  
   When $m=s=1$, we have
\begin{align*}
\sum_{k=0}^{\frac{p-1}2} \binom{2k}{k}^2\left(\frac{\l}{16}\right)^k
& \equiv \sum_{k=0}^{\frac{p-1}2} \binom{\frac{p-1}2+k}{k}\binom{\frac{p-1}2}{k}(-\l)^k & \pmod{p^2} \\
& \equiv \left( \frac{\l - 1}{p}\right) \left( \frac{1}{\l-1}\right)^{\frac{p-1}{2}} \sum_{k=0}^{\frac{p-1}2} \binom{\frac{p-1}2}{k}^2 \l^{k} & \pmod {p^2}.
\end{align*}
These  generalize the supercongruences in \cite{CLZ} with $\l=2$ and
as well as one of Mortenson's supercongruences with $\l =1$, in which $L_\l$ degenerates \cite{Mort}.

Theorem \ref{maintheorem} relates values of the truncated hypergoemetric function $F_3(\l)_n$ to the family of K3 surfaces
$S_\lambda: z^2=xy(x-1)(y-1)(x-\lambda y)$ and the family of elliptic curves $E_\l: y^2=(x-1)(x^2-\frac{1}{1-\l})$.  Let  $A_p(\lambda)=\#S_\lambda(\mathbb{F}_p)-p^2-1$. 
Then $F_3(\l)_{\frac{p-1}{2}} =  A_p(\l) \pmod p$.
The variation of the complex structure of the family $S_\lambda$ of K3 surfaces is again depicted by its Picard-Fuchs differential equation, which is projectively equivalent to the symmetric square of
the Picard-Fuchs equation of  $E_\lambda: y^2=(x-1)(x^2-\frac 1{1-\lambda})$ \cite{Long}.
In terms of arithmetic,  if we let
$a_p(\lambda)=p+1-\#E_\lambda(\mathbb{F}_p)$, then $A_p(\lambda)=\left(\frac{1-\l}{p}\right)(a_p(\lambda)^2-p)$ \cite{AOP}.



Specializing formula (1) of
\cite{CosterVanHamme} to  $E_\l$ and using uniformizer $t= \frac{x-1}{y}$ as a local parameter at the
point at infinity, we obtain the expansion
\begin{equation}\label{omegaexpansion2}
 -\frac{dx}{2y} = \sum_{k=0}^\infty P_k\left( \sqrt{1-\lambda} \right) \left(\frac{2}{\sqrt{1-\l}}\right)^k t^{2k+1} \frac{dt}{t}.
 \end{equation}
The coefficients of this differential form on $E_\l$ satisfy ASD congruences, but our interest is in supercongruences for hypergeometric 
functions $_3F_2$, which we obtain using Z.-H. Sun's identity (1.7) \cite{Sun} (with
$x= - \frac{\lambda}{4}$) 
\[\sum_{k=0}^{n} \binom{2k}{k}^2 \binom{n+k}{2k} \left(- \frac{\lambda}{4}\right)^k = 
P_{n}(\sqrt{1-\l})^2. \]

With the identity
\begin{equation} \label{ZHSunidentity}
 _3F_2\left(\begin{smallmatrix}\tfrac{1}{2}, \; \tfrac{1-n}{2}, \; \tfrac{1 + n}{2} \\ 1, \; 1 \end{smallmatrix}; \l \right)
=\sum_{k=0}^n \binom{2k}{k}^2 \binom{n+k}{2k} \left(\frac{-\l}{4}\right)^k  \end{equation}
along with Theorem \ref{CosterVanHammeTheorem}, we obtain the following supercongruences.

\begin{proposition} For CM values $\l$ of the family $E_\l: y^2=(x-1)(x^2-\frac{1}{1-\l})$, such that $\l \in \mathbb{Z}_p$ and $p$ is ordinary,
for all positive integers $m$ and $s$ with $m$ odd,
\begin{align*}
& _3F_2\left(\begin{smallmatrix}\tfrac{1}{2}, \; \tfrac{1-mp^s}{2}, \; \tfrac{1 + mp^s}{2} \\ 1, \; 1 \end{smallmatrix}; \l \right) \\
& \equiv  \left( \left( \frac{\l-1}{p} \right)\alpha_{p,\lambda}{}^2\right)   {}
_3F_2\left(\begin{smallmatrix}\tfrac{1}{2}, \; \tfrac{1-mp^{s-1}}{2}, \; \tfrac{1 + mp^{s-1}}{2} \\ 1, \; 1 \end{smallmatrix}; \l \right)
  \pmod{p^{2s}}
\end{align*}
where $\alpha_{p,\l}$ is the unit root of $X^2-[p+1-\#(E_\l/{\mathbb F}_p)]X+p=0$.
\end{proposition}

Turning our attention to $S_\l$ and the sequence $F_3(\l)_n $, we use the 
congruence $\binom{\frac{p-1}{2}+k}{2k} \equiv \bigl(\frac{-1}{16}\bigr)^k \binom{2k}{k} \pmod{p^2}$ and the equality 
$\frac{(1/2)_k}{k!} = \binom{2k}{k} \frac{1}{4^k}$ in \eqref{ZHSunidentity} to obtain 
\[ F_3(\l)_{\frac{p-1}{2}} := {} _3F_2\left ( \begin{array}{cccc}1/2,&1/2,&1/2 \\ &1&1\end{array};\l \right )_{\frac{p-1}2}
\equiv P_{\frac{p-1}{2}}(\sqrt{1-\l})^2 \pmod{p^2} .\]
Notice that this can also be obtained directly from Clausen's formula expressing certain values of $_3F_2$ 
as the square of values of $_2F_1$.




\begin{pot}
We begin by observing that 
\[ P_{\frac{p-1}{2}}\left( \sqrt{1-\lambda} \right)  \equiv 
\left\{ \begin{array}{lll}  0 &\hspace{-0.2in} \pmod{p}  \qquad & \text{if }p\text{ is supersingular}
\\ \varepsilon \alpha_{p,\l} &\hspace{-0.2in} \pmod{p^2} \qquad & \text{if }p\text{ is ordinary} \end{array} \right.  \]
where $\varepsilon$ is a fourth root of unity.
The congruence in the supersingular case is just the first instance of the standard ASD congruences.  Since we are assuming 
$\l$ is a CM value of $E_\l$, we have supercongruences at ordinary primes by 
Theorem \ref{CosterVanHammeTheorem}.
By the ASD congruences, we may conclude that $\varepsilon$ is the fourth root of unity in $\mathbb{Q}_p$, or
possibly in its unramified quadratic extension, that is congruent to
$\left(\frac{\sqrt{1-\l}}{2}\right)^{\frac{p-1}{2}}$ modulo $p$.
Squaring these congruences for $P_{\frac{p-1}{2}}\left( \sqrt{1-\lambda} \right)$ immediately yields Theorem \ref{maintheorem}:
$$_3F_2\left ( \begin{array}{c}\frac{1}{2}, \frac{1}{2}, \frac{1}{2}  \\  1, 1\end{array}; \l \right )_{\frac{p-1}2}=\sum_{k=0}^{\frac{p-1}{2}}
\left(\frac{(\frac 12)_k}{k!}\right)^3\l^k\equiv \left( \frac{\l-1}{p} \right)\alpha_{p,\lambda}{}^2 \pmod {p^2}.$$
Many of the intermediate statements in this argument involve choosing an embedding of $\sqrt{1-\lambda}$ in $\mathbb{Q}_p$ or its unramified quadratic
extension, but this choice does not affect the square of the congruence.
\end{pot}

We note that this establishes,   
modulo $p^2$, all cases of Conjecture 5.2 of \cite{ZWSun3} by Z.W. Sun.  
These conjectures can be written as  
$$ \sum_{k=0}^{\frac{p-1}{2}} \binom{2k}{k}^3  \left(\frac{\l}{64}\right)^k \equiv 
\left\{ \begin{matrix}  \left(\frac{c}{p}\right)(4a^2-2p) \pmod{p^2}  \hfill
& \quad {\text{if }} (\frac{p}{D})=1 
{\text{ where }} a^2 + Db^2 = p 
\\ 0 \pmod{p^2} \hfill & \quad {\text{if }} (\frac{p}{D})=-1 \hfill \end{matrix} \right. , $$
with appropriate choices of $D \in \mathbb{Z}_+$ and character $\left(\frac{c}{p}\right)$.  Note that 
$\sum_{k=0}^{\frac{p-1}{2}} \binom{2k}{k}^3  \left(\frac{\l}{64}\right)^k = {}
_3F_2\left (\begin{array}{c}
\frac{1}{2},\frac{1}{2},\frac{1}{2}\\
1,1 \end{array}; \l \right )_{\frac{p-1}{2}}$ 
 via the identity 
$\frac{(1/2)_k{}^3}{k!^3} = \binom{2k}{k}^3 \frac{1}{64^k}$.  
These conjectures address the $\l$-values $\l = -8, 1, -\frac{1}{8}, 4, \frac{1}{4}, 64, \frac{1}{64}, -1$, which are 
all of the CM values for $E_\l$ over $\mathbb{Q}$, as verified in \cite{AOP}, with the exception of the degenrate case $\l=1$, 
for which $E_\l$ is not an elliptic curve.  The supercongruence for $\l=1$ was proved by Van Hamme in \cite{VH} and by Ono 
in \cite{Ono}. 

If $E_\l$ has CM over $K=\Q(\sqrt{-D})$, then the attached 2-dimensional
representation $\rho$ decomposes into 2 Grossencharacters when $\rho$ is
restricted to ${\rm{Gal}}(\overline{\mathbb{Q}}/K)$. Then at splitting
primes $p$, which are precisely the ordinary primes of $E_\l$, the trace of the Frobenius is 
$A_p=\alpha_p+\beta_p$, where both $\alpha_p$ and $\beta_p$ are in
the ring of integers of the quadratic field K and have the same
absolute value $\sqrt{p}$. In the case that $K$ has class number 1, (all Sun $\l$ values
correspond to class number 1 cases), then ideals ($\alpha_p$) and
($\beta_p$) are the two distinct prime ideals above $p$. That is,
$\alpha_p=a+b\sqrt{-D}$ and $\beta_p = a-b\sqrt{-D}= \frac{p}{\alpha_p}$, where $a$ and $b$ are integers or half integers depending on
$p \equiv 1 \text{ or } 3 \pmod 4$, such that $a^2+b^2D=p$.  Our congruences involve
$\alpha_p^2$, which is just $a^2-Db^2+2ab\sqrt{-D}$.  Using $\beta_p^2 = a^2 -Db^2 -2ab\sqrt{-D} \equiv 0 \pmod{p^2}$ and $a^2 +b^2D =p$, 
we have $\alpha_p^2 \equiv 4a^2-2p \pmod{p^2}$, which, along with the character $\left( \frac{1-\l}{p} \right)$, is the target 
of Z.-W. Sun's congruences.  In the supersingular case, we simply have $\alpha_p=0$.  

Alternately, we note that Ono has explicitly identified the values $\alpha_p$, for all CM curves $E_\l$ with $\l \in \mathbb{Z}$, 
in Theorem 6 of \cite{Ono}.  These values $\alpha_p$ determine the formal group structure and the ASD congruences (i.e., that 
$a_p \equiv \left(\frac{1-\l}{p}\right) \alpha_p^2 \pmod{p}$);
combining this with Coster and Van Hamme's supercongruences gives another proof of Sun's conjectures, that 
$a_p \equiv \left(\frac{1-\l}{p}\right) \alpha_p^2 \pmod{p^2}$.

Note that we have established just the first interesting supercongruence for the truncated hypergeometric functions 
$_3F_2\left( \begin{smallmatrix}\tfrac{1}{2}, \; \tfrac{1}{2}, \; \tfrac{1}{2} \\ 1, \; 1 \end{smallmatrix}; \l \right)_{\tfrac{k-1}{2}}$, when $k=p$;
we have congruence modulo $p^2$ where the standard ASD congruences only guarantee congruence modulo $p$.
However, when $\lambda \in \mathbb{Q}$ is a CM value, these congruences actually appear to hold modulo $p^3$ for odrinary primes $p$.
Further, we expect the whole infinite family of
supercongruences to hold for
$_3F_2\left( \begin{smallmatrix}\tfrac{1}{2}, \; \tfrac{1}{2}, \; \tfrac{1}{2} \\ 1, \; 1 \end{smallmatrix}; \l \right)_{\tfrac{p-1}{2}}$:
\begin{conjecture} \label{maintheoremconj}
For CM values $\l$ of the family $E_\l: y^2=(x-1)(x^2-\frac{1}{1-\l})$, such that $\l \in \mathbb{Z}_p$ and $p$ is ordinary,
for all positive integers $m$ and $s$ with $m$ odd,
\begin{align*}
& _3F_2\left ( \begin{matrix}\tfrac{1}{2},&\tfrac{1}{2},&\tfrac{1}{2} \\ &1&1\end{matrix};\l \right )_{\frac{mp^s-1}2} \\
& \equiv \left( \left( \frac{\l-1}{p} \right)\alpha_{p,\lambda}{}^2 \right) {}
_3F_2\left ( \begin{matrix} \tfrac{1}{2},&\tfrac{1}{2},&\tfrac{1}{2}\\ &1&1\end{matrix};\l \right )_{\frac{mp^{s-1}-1}2} \pmod{p^{2s}}
\end{align*}
where $\alpha_{p,\l}$ is the unit root of $X^2-[p+1-\#(E_\l/{\mathbb F}_p)]X+p=0$.
\end{conjecture}

\section{Corollaries} \label{corollarysec}

An idea of Gessel for dealing with the supercongruences of 
the Ap\'ery numbers \[ c_n=\, _{4}F_{3}\left(\begin{array}{c}
-n,-n,1+n,1+n \\ 1,\; 1,\; 1 \end{array};1\right)=\sum_{k=0}^n \binom{n}{k}^2 \binom{n+k}{k}^2\] is as follows.
He identified the auxiliary sequence $ b_n = 2 \sum_{k=0}^n \binom{n}{k}^2 \binom{n+k}{k}^2 (H_{n+k}
-H_{n-k}),$ where $H_k$ is the harmonic sum $\sum_{j=1}^k \frac 1j$,
and showed that $c_{k+pn} \equiv (c_k+pnb_k)c_n \pmod{p^2}$ where $0 \leq k < p$ \cite{Gessel}.  Using the idea of Ishikawa \cite{Ishikawa}, we take
$k=n=\frac{p-1}{2}$. It follows that when $c_{(p-1)/2}\not \equiv 0 \pmod p$,
 we have the supercongruence $c_{(p^2-1)/2} \equiv c_{(p-1)/2}^2 \pmod{p^2}$, since
$b_{(p-1)/2}  \equiv 0 \pmod{p}$ from the $p$-adic properties of Harmonic sums.
In \cite{AO}, Ahlgren and Ono also need an
entity similar to $b_{(p-1)/2}$ to be zero modulo $p$, which they established using a binomial coefficient identity proved by the WZ method \cite{AEO}.

In the above examples, supercongruences of a sequence $c_n$ were shown to be equivalent to congruences of an auxilliary sequence $b_n$; 
and the congruences for $b_n$ were proved using whatever method applied in each case.  Similarly, 
the supercongruence in Theorem \ref{maintheorem} for the sequence $a_n=\sum_{i=0}^n \binom{2i}{i}^3 (\frac{\l}{64})^i$ is equivalent 
to the auxiliary congruence in Corollary \ref{maincorollary} for the sequence $b_n = 
\sum_{i=0}^n \binom{2i}{i}^3 (\frac{\l}{64})^i(6(H_{2i}-H_i) + \frac{(\l/64)^{p-1}-1}{p})$.  However,  we 
proved our supercongruence using the theorem of Coster and Van Hamme, and thus obtain our auxilliary congruence.  We know of 
no direct proof of Corollary \ref{maincorollary}; we expect a proof for each fixed individual $\l$ might
require some combinatorial identity and additional intelligent guesses of WZ pairs to prove the identity, see \cite{Ahlgren, AO}.

\begin{lemma} \label{aux}
For the sequence $a_n=\sum_{i=0}^n \binom{2i}{i}^3 (\frac{\l}{64})^i$, we introduce the auxilliary sequence $b_n = 
\sum_{i=0}^n \binom{2i}{i}^3 (\frac{\l}{64})^i(6(H_{2i}-H_i) + \frac{(\l/64)^{p-1}-1}{p})$. Then for any prime $p$, any $k$ with 
$\frac{p-1}{2}\leq k <p$, and any $n$, 
\[a_{k+pn} \equiv a_k a_n + p  b_k \sum_{i=0}^n i \binom{2i}{i}^3 \left(\frac{\l}{64}\right)^i  \pmod{p^2}. \]
\end{lemma}

\begin{proof}
Notice we can write $a_{k+pn} - a_k a_n$ as the telescoping sum $\sum\limits_{i = 1}^{n}T_{k,i}$, where
\begin{align*}
T_{k,n} &= (a_{k+pn} - a_k a_n) - (a_{k+p(n-1)} - a_ka_{n-1})\\
&= (a_{k+pn} - a_{k+p(n-1)}) - a_{k}(a_n - a_{n-1})\\
&= \sum\limits_{i=-p+k+1}^{k} \binom{2i+2pn}{i+pn}^3 \left( \frac{\lambda}{64} \right)^{i+pn} - 
\left( \sum\limits_{i=0}^{k} \binom{2i}{i}^3  \left( \frac{\lambda}{64} \right)^{i} \right) \binom{2n}{n}^3 \left( \frac{\lambda}{64} \right)^{n} 
\end{align*}
Using the condition that $\frac{p-1}{2}\leq k <p$, we notice that $\binom{2i+2pn}{i+pn} \equiv 0 \pmod{p}$ if $-p+k+1 <i<0$.  
Simplifying modulo $p^2$, these terms disappear and we can factor.

\begin{align*}
T_{k,n} &\equiv \sum\limits_{i=0}^{k}\left( \binom{2i+2pn}{i+pn}^3 
\left( \frac{\lambda}{64} \right)^{pn} - \binom{2n}{n}^3 \left( \frac{\lambda}{64} \right)^{n} \binom{2i}{i}^3 \right) 
\left( \frac{\lambda}{64} \right)^{i} \pmod {p^2}
\end{align*}

The factor $\binom{2i+2pn}{i+pn}^3$ may be rewritten as $ \frac{-\Gamma_p(1+2i+2pn)^3}{\Gamma_p(1+i+pn)^6}  \binom{2n}{n}^3$, 
where $\Gamma_p$ is the $p$-adic gamma function (see Chapter 11 \cite{Ono}). Let 
$T_{k,n} \equiv \left( \frac{\lambda}{64} \right)^n \binom{2n}{n}^3 U_{k,n} \pmod {p^2}$, where 
$$U_{k,n} = \sum\limits_{i=0}^{k} \left( \left( \frac{-\Gamma_p(1+2i+2pn)^3}{\Gamma_p(1+i+pn)^6} \right) 
\left( \frac{\lambda}{64} \right)^{(p - 1)n} - \binom{2i}{i}^3 \right) \left( \frac{\lambda}{64} \right)^{i}.$$ 
To simplify the $p$-adic gamma function modulo $p^2$, we expand $\Gamma_p$ in terms of factorials and 
harmonic sums $H_n=\sum_{i=1}^n \frac{1}{i}$.  (By convention, $H_0=0$.)  We also use the congruence, for $p>3$, that
$H_{p-1} \equiv 0 \pmod {p}$.  (Wolstenholme has shown this congruence holds modulo $p^2$, though we only need modulo $p$.)
\begin{align*}
\Gamma_p(1+i+pn)^r &\equiv (-1)^{(1+i+pn)r}i!^r (1 + pnrH_i)\prod\limits_{j = 0}^{n - 1}(p-1)!^r(1 + pjrH_{p-1}) \pmod {p^2}\\
&\equiv (-1)^{(1+i+pn)r}i!^r (1 + pnrH_i)(-1)^{nr} \pmod {p^2}\\
&\equiv (-1)^{(1+i)r}i!^r (1 + pnrH_i) \pmod {p^2}
\end{align*}

Plugging this into $U_{k,n}$, we have

\begin{align*}
U_{k,n} &\equiv \sum\limits_{i=0}^{k} \left( \left( \frac{(2i)!^3 (1 + 6pnH_{2i}) }{(i)!^6 
(1 + 6pnH_{i})} \right) \left( \frac{\lambda}{64} \right)^{(p - 1)n} - \binom{2i}{i}^3 \right) \left( \frac{\lambda}{64} \right)^{i} \pmod {p^2} \\
 &\equiv \sum\limits_{i=0}^{k} \binom{2i}{i}^3 \left( \frac{\lambda}{64} \right)^i  
 \left( ( 1+ 6pn(H_{2i} - H_i)) \left( \frac{\lambda}{64} \right)^{(p - 1)n} - 1 \right) \pmod {p^2}
\end{align*}

Using $\left( \frac{\lambda}{64} \right)^{(p - 1)n} = \left( 1 + p \left( \frac{\left( \frac{\lambda}{64} \right)^{p - 1} - 1}{p} \right) \right)^n
\equiv 1 +  p n \left( \frac{\left( \frac{\lambda}{64} \right)^{p - 1} - 1}{p} \right) \pmod {p^2}$, 

\begin{align*}
U_{k,n} &\equiv \sum\limits_{i=0}^{k} \binom{2i}{i}^3 \left( \frac{\lambda}{64} \right)^i 
\left( \left( 1+ 6pn(H_{2i} - H_i)\right)\left(1+  pn \left( \frac{\left( \frac{\lambda}{64} \right)^{p - 1} - 1}{p}\right)\right)-1 \right) \\
&\equiv pn\sum\limits_{i=0}^{k} \binom{2i}{i}^3 \left( \frac{\lambda}{64} \right)^i  \left( 6(H_{2i} - H_i) + 
\left( \frac{\left( \frac{\lambda}{64} \right)^{p - 1} - 1}{p} \right) \right) \pmod {p^2}
\end{align*}

So $T_{k,n} \equiv pn \binom{2n}{n}^3 (\frac{\l}{64})^n b_k \pmod{p^2}$.  Combining this congruence with the 
telescoping sum  $a_{k+pn} - a_k a_n=\sum\limits_{i = 1}^{n}T_{k,i}$ completes the proof of the lemma.
\end{proof}

Using this lemma, 
we show the equivalence of Theorem \ref{maintheorem} and Corollary \ref{maincorollary}.

\begin{poc}
We consider $T_{k,n}$ with $k = \frac{p-1}{2}$ and $n=1$.
By definition, $T_{\frac{p-1}{2},1} = a_{\frac{3p-1}{2}} - a_{\frac{p-1}{2}}a_{\frac{3-1}{2}}$; we can 
rewrite this, modulo $p^2$, as $P_{\frac{3p-1}{2}}(\sqrt{1-\l})^2 - P_{\frac{p-1}{2}}(\sqrt{1-\l})^2P_{\frac{3-1}{2}}(\sqrt{1-\l})^2$.  Since 
the sequence $P_{\frac{n-1}{2}}(\sqrt{1-\l})$ satisfies ASD congruences, we know that $T_{\frac{p-1}{2},1} \equiv 0 \pmod{p}$.  However, 
Theorem \ref{maintheorem} is precisely the information we need to conclude that $T_{\frac{p-1}{2},1} \equiv 0 \pmod{p^2}$ whenever 
$\l$ is a CM value of $E_\l$ that embeds in $\mathbb{Z}_p$.  

Thus, since
\[T_{\frac{p-1}{2},1} \equiv \frac{p\l}{8} \sum\limits_{i=0}^{(p-1)/2} \binom{2i}{i}^3 \left( \frac{\lambda}{64} \right)^i
\left( 6(H_{2i} - H_i) + \left( \frac{\left( \frac{\lambda}{64} \right)^{p - 1} - 1}{p} \right) \right) \pmod{p^2},\]
we have the desired congruence $b_{\frac{p-1}{2}} \equiv 0 \pmod{p}$ whenever we have supercongruences for $a_{\frac{p-1}{2}}$.

\end{poc}

\section{Acknowledgements}

The paper grew out of the Research Experience for  Undergraduates in number theory at Iowa State, which took place in Fall 2011.
It was supported by NSF grant DMS-1001332. A significant part of the work was done while the second author was visiting Cornell University as an AWM Ruth I.
Michler fellow. She would like to thank both AWM and Cornell for the excellent opportunity.  The authors thank Zhi-Wei Sun for pointing out his papers 
\cite{ZWSun3, ZWSunNote}, which contain his survey of the subject and many of the conjectures studied here, and for pointing 
out \cite{VH}.  The authors also thank Ravi Ramakrishna for his helpful comments.

\end{document}